\documentclass[11pt]{amsart}
\usepackage[T1]{fontenc}
\usepackage{lmodern}
\usepackage{amsmath,amssymb,amsthm}
\usepackage[margin=1.08in]{geometry}
\usepackage{microtype}
\usepackage{comment}
\usepackage[hidelinks]{hyperref}
\hypersetup{pdftitle={BEURLING AND MODEL SUBSPACES INVARIANT UNDER A UNIVERSAL OPERATOR},pdfauthor={Ben Hur Eidt and S. Waleed Noor}}
\newcommand{\D}{\mathbb D}
\newcommand{\T}{\mathbb T}

\newcommand{\Nzero}{\mathbb N_0}
\newcommand{\Hp}{H^2}
\newcommand{\Hinf}{H^\infty}
\newcommand{\Schur}{\mathcal S}
\newcommand{\Ces}{\mathcal C}

\DeclareMathOperator{\supp}{supp}
\DeclareMathOperator{\spanop}{span}

\newtheorem{maintheorem}{Theorem}

\newtheorem{theorem}{Theorem}[section]
\newtheorem{proposition}[theorem]{Proposition}
\newtheorem{lemma}[theorem]{Lemma}
\newtheorem{corollary}[theorem]{Corollary}
\theoremstyle{definition}
\newtheorem{example}[theorem]{Example}
\theoremstyle{remark}

\numberwithin{equation}{section}
\title[Beurling and model subspaces invariant under a universal operator]{BEURLING AND MODEL SUBSPACES INVARIANT UNDER A\\ UNIVERSAL OPERATOR}
\author{Ben Hur Eidt}
\address{ICMC, Universidade de São Paulo, São Carlos, SP, Brazil}
\email{benhur@icmc.usp.br}
\author{Waleed Noor}
\address{IMECC, Universidade Estadual de Campinas, Campinas, SP, Brazil}
\email{waleed@unicamp.br}

\date{}
\begin{document}
\begin{abstract}
We classify the Beurling and model subspaces of the Hardy space
$H^2(\mathbb D)$ that are invariant under the composition operators
$C_{\phi_a}$, where $\phi_a(z)=az+1-a$ and $0<a<1$. These operators
have universal translates in the sense of Rota and provide a
function-theoretic approach to the Hilbert space Invariant Subspace
Problem. For each fixed $a$, a Beurling subspace $\Theta H^2$ is
invariant precisely when
\[
\Theta(z)=\lambda B(z)\exp\left(-c\,\frac{1+z}{1-z}\right),
\qquad |\lambda|=1,\quad c\ge0,
\]
where $B$ is a possibly constant Blaschke product satisfying
$m_B(w)\le m_B(\phi_a(w))$ for every $w\in\mathbb D$, with $m_B$
denoting zero multiplicity. For every affine self-map
$\phi(z)=\alpha z+\beta$ with $0<|\alpha|<1$, we show that the
nonzero invariant model spaces are exactly the spaces of
polynomials of bounded degree. We also recover the corresponding
Beurling and model subspace classifications for the classical
Ces\`aro operator. Finally, we characterize when a cyclic invariant
subspace contains a nonzero function analytic across $1$ through
its intersections with model spaces whose boundary spectra omit
$1$, and exhibit an infinite-dimensional cyclic subspace containing
an eigenvector but no such analytic function.
\end{abstract}
\maketitle

\section{Introduction}
Let $\Hp=H^2(\D)$ be the Hardy space on the open unit disk, and let $C_\phi f=f\circ\phi$ be the bounded composition operator induced by an analytic self-map $\phi$ of $\D$. For an inner function $\Theta$, the Beurling subspace $\Theta\Hp$ and the model space
\[
K_\Theta=\Hp\ominus\Theta\Hp
\]
are invariant under the unilateral shift $(Sf)(z)=zf(z)$ for $f\in H^2$ and its adjoint, respectively. We ask which of these spaces are also invariant under composition operators with symbols from a family of affine self-maps of the unit disk $\mathbb{D}$ defined by
\begin{equation}\label{eq:family}
\phi_a(z)=az+1-a,\qquad 0<a<1.
\end{equation}
These maps satisfy $\phi_a\circ\phi_b=\phi_{ab}$ and have a common attractive fixed point at $1$. Carmo and Noor \cite{CN} proved that $C_{\phi_a}$ has universal translates in the sense of Rota \cite{Rota}. Subtracting a scalar multiple of the identity does not change
the invariant subspaces. Consequently, for any fixed $a\in(0,1)$, the Invariant Subspace Problem (ISP) for complex separable infinite-dimensional Hilbert spaces is equivalent to the assertion that every minimal nonzero invariant subspace of $C_{\phi_a}$ is one-dimensional. Further related results were obtained by Carmo, Eidt and Noor \cite{CEN}.

The quotient criterion of Bose, Muthukumar and Sarkar~[1]
and the affine model-space criterion of Muthukumar and
Sarkar~[11] provide the starting point for our classifications.
For every affine self-map $\varphi(z)=\alpha z+\beta$ of
$\mathbb D$ with $0<|\alpha|<1$, we prove that the only
nonzero model spaces invariant under $C_\varphi$ are
\[
K_{z^n}=\operatorname{span}\{1,z,\ldots,z^{n-1}\},
\qquad n\ge1.
\]
For $\varphi_a(z)=az+1-a$, with $0<a<1$, we show that the
singular measure of an inner function $\Theta$ defining a
$C_{\varphi_a}$-invariant Beurling subspace $\Theta H^2$
must be concentrated at $1$, even when $\Theta$ has zeros.
Together with the Blaschke multiplicity
criterion, this gives a complete classification.

The corresponding Ces\`aro classifications follow from
the semigroup correspondence of Gallardo-Guti\'errez and
Partington~[7]; the model-space case was already established
by Gallardo-Guti\'errez, Partington and Ross~[8].
Finally, Section~6 shows that a cyclic $C_{\varphi_a}$-invariant
subspace contains a nonzero function analytic across $1$
precisely when it intersects a model space whose boundary
spectrum omits $1$ nontrivially. We also construct an
infinite-dimensional cyclic subspace containing an
eigenvector but having trivial intersection with every
such model space.

\begin{maintheorem}\label{thm:model}
Let $\phi(z)=\alpha z+\beta$ be a self-map of $\D$ with $0<|\alpha|<1$. For an inner function $\Theta$, the following are equivalent:
\begin{enumerate}
\item $C_\phi K_\Theta\subseteq K_\Theta$;
\item $\Theta(z)=\lambda z^n$ for some $\lambda\in\T$ and $n\in\Nzero$.
\end{enumerate}

\end{maintheorem}

The atomic singular inner function with a single atom at \(1\) is
\begin{equation}\label{eq:atomic}
S_c(z)=\exp\left(-c\frac{1+z}{1-z}\right),\qquad c\ge0,
\end{equation}
with associated singular measure $c\delta_1.$ For a nonzero analytic function $F$, let $m_F(w)$ denote its zero multiplicity at $w$, with $m_F(w)=0$ when $F(w)\ne0$. Then our Beurling space classification is
\begin{maintheorem}\label{thm:beurling}
Fix $a\in(0,1)$. For an inner function $\Theta$, the space $\Theta\Hp$ is $C_{\phi_a}$-invariant if and only if
\begin{equation}\label{eq:classification}
\Theta=\lambda B S_c,\qquad \lambda\in\T,\quad c\ge0,
\end{equation}
where $B$ is a possibly constant Blaschke product satisfying
\begin{equation}\label{eq:multiplicity}
m_B(w)\le m_B(\phi_a(w))\qquad(w\in\D).
\end{equation}
\end{maintheorem}
The multiplicity criterion for Blaschke products is due to Bose, Muthukumar and Sarkar \cite{BMS}, building on earlier work of Cowen and Wahl \cite{CW}. The invariance of $S_c\Hp$ for maps with boundary Denjoy--Wolff point $1$ is also contained in \cite{CW}. 
We then apply Theorems A and B to  classify invariant model and Beurling subspaces for the \emph{Ces\`aro operator} on $H^2$ 
\[
(\Ces f)(z)=\sum_{n=0}^\infty
\left(\frac1{n+1}\sum_{k=0}^n f_k\right)z^n,
\qquad f(z)=\sum_{k=0}^\infty f_kz^k
\] 
using a correspondence result established by Gallardo-Guti\'errez and Partington \cite{GP}: for any closed subspace $M\subseteq\Hp$ we have
\begin{equation}\label{eq:cesarocorrespondence}
\Ces M\subseteq M
\quad\Longleftrightarrow\quad
C_{\phi_a}M^\perp\subseteq M^\perp\quad\text{for every }0<a<1.
\end{equation}

\begin{maintheorem}\label{thm:cesaro}
Let $\Theta$ be an inner function. Then
\begin{enumerate}
\item $\Theta\Hp$ is $\Ces$-invariant if and only if $\Theta=\lambda z^n$ for some $\lambda\in\T$ and $n\in\Nzero$.
\item $K_\Theta$ is $\Ces$-invariant if and only if $\Theta=\lambda S_c$ for some $\lambda\in\T$ and $c\ge0$.
\end{enumerate}
As a consequence, there are no minimal or maximal $\Ces$-invariant model subspaces, and no minimal $\Ces$-invariant Beurling subspaces.
\end{maintheorem}

\section{Preliminaries}\label{sec:prelim}
We use $\Nzero=\{0,1,2,\ldots\}$ and write $m$ for normalized Lebesgue measure on $\T$. An inner function is a bounded analytic function in the unit disk $\mathbb{D}$ with unimodular radial limits almost everywhere on $\mathbb{T}$. The canonical factorization of an inner function is
\[
\Theta=\lambda B S_\mu,\qquad
S_\mu(z)=\exp\left(-\int_{\T}\frac{\xi+z}{\xi-z}\,d\mu(\xi)\right),
\]
where $\mu$ is a finite positive measure singular with respect to $m$, and the zeros of $B$, counted with multiplicity, satisfy the Blaschke condition. We allow $B=1$ and $\mu=0$. For background on Hardy spaces and the Beurling and model subspaces we use \cite{FM,MR}. The \emph{boundary spectrum} of an inner function $\Theta$ is defined by
\begin{equation}\label{eq:spectrum}
\Sigma(\Theta)=\bigl(\overline{Z(\Theta)}\cap\T\bigr)\cup\supp\mu,
\end{equation}
where $Z(\Theta)$ denotes the zero-set of $\Theta$ and its closure is taken in $\overline\D$. This is the set of boundary points across which $\Theta$ does not extend analytically. Every function in $K_\Theta$ extends analytically across $\T\setminus\Sigma(\Theta)$; see \cite[Lemma 14.27, Theorem 14.28]{FM}. 

Let $\psi$ be a nonconstant analytic self-map of $\mathbb{D}$. The quotient criterion of Bose, Muthukumar and Sarkar \cite{BMS}, together with the equivalent $H^\infty$ formulation in \cite{MS}, is
\begin{equation}\label{eq:quotient}
C_\psi(\Theta\Hp)\subseteq\Theta\Hp
\quad\Longleftrightarrow\quad
\frac{\Theta\circ\psi}{\Theta}\in\Hinf
\quad\Longleftrightarrow\quad
\frac{\Theta\circ\psi}{\Theta}\in\Schur,
\end{equation}
where $\Schur=\{q\in\Hinf:\|q\|_\infty\le1\}$ is the \emph{Schur} class. In particular, for a Blaschke product,
\begin{equation}\label{eq:blaschkecriterion}
C_\psi(B\Hp)\subseteq B\Hp
\quad\Longleftrightarrow\quad
m_B(w)\le m_{B\circ\psi}(w)\quad(w\in\D).
\end{equation}
For affine $\phi(z)=\alpha z+\beta$, $\alpha\ne0$, Muthukumar and Sarkar \cite[Theorem~4.3]{MS} prove
\begin{equation}\label{eq:modelcriterion}
C_\phi K_\Theta\subseteq K_\Theta
\quad\Longleftrightarrow\quad
\frac{\Theta\circ\sigma}{\Theta}\in\Hinf,
\qquad
\sigma(z)=\frac{\overline\alpha z}{1-\overline\beta z}.
\end{equation}
These established criteria will be used throughout this work. 

\section{Invariant model spaces}\label{sec:models}
\begin{proof}[Proof of Theorem~\ref{thm:model}]
Suppose $K_{z^n}=\spanop\{1,z,\ldots,z^{n-1}\}$ for $n\ge1$. Composition with an affine map cannot increase polynomial degree, hence these spaces are trivially invariant under $C_{\alpha z+\beta}$. Conversely suppose $C_\phi K_\Theta\subseteq K_\Theta$. By the criterion of Muthukumar and Sarkar, there is a function $q\in\Hinf$ such that
\begin{equation}\label{eq:modelidentity}
\Theta\circ\sigma=\Theta q,
\qquad \sigma(z)=\frac{\overline\alpha z}{1-\overline\beta z}.
\end{equation}
We first locate the zeros of $\Theta$. Since $\phi(z)=\alpha z+\beta$ maps $\D$ into itself, $|\alpha|+|\beta|\le1$. For $z\in\D$,
\[
|\sigma(z)|\le\frac{|\alpha||z|}{1-|\beta||z|}
\le |z|\frac{|\alpha|}{1-|\beta|}\le |z|.
\]
Thus $\sigma$ is a self-map of $\D$ fixing $0$; it is injective because $\alpha\ne0$. It satisfies $|\sigma'(0)|=|\alpha|<1$. The strict form of Schwarz's lemma gives
\[
|\sigma(z)|<|z|\quad(0<|z|<1),
\qquad \sigma^{\circ j}(z)\longrightarrow0.
\]
Indeed, the moduli of the iterates decrease; if their limit were positive, continuity would contradict the strict inequality on a compact circle about $0$. If $\Theta(w)=0$, identity \eqref{eq:modelidentity} says that $\Theta(\sigma(w))=0$. Thus every forward iterate of $w$ is a zero. When $w\ne0$, these iterates remain nonzero because $\sigma$ is injective and $\sigma(0)=0$. Their moduli strictly decrease, so they are distinct; they tend to the interior point $0$. The identity theorem would force $\Theta\equiv0$, which an inner function cannot be. Hence $0$ is the only possible zero. Factoring it with its multiplicity gives
\[
\Theta=\lambda z^nS,\qquad n\in\Nzero,
\]
where $S$ is a zero-free inner function. We now show that $S$ is constant. Substitution into \eqref{eq:modelidentity} gives
\[
q(z)=\left(\frac{\overline\alpha}{1-\overline\beta z}\right)^n g(z),
\qquad g(z):=\frac{S(\sigma(z))}{S(z)}.
\]
The first factor and its reciprocal belong to $\Hinf$, so $g\in\Hinf$. At almost every boundary point, the radial limit of $S$ has modulus one, while $|S(\sigma(z))|\le1$ inside the disk. Taking radial limits in $S\circ\sigma=Sg$ gives $|g^*|\le1$ almost everywhere. The boundary norm theorem for $\Hinf$ implies $\|g\|_\infty\le1$. Because $S$ has no zeros and $\sigma(0)=0$, we also have $g(0)=S(0)/S(0)=1$. The maximum modulus principle forces $g\equiv1$. It follows that $S\circ\sigma=S$. Iterating this identity and using $\sigma^{\circ j}(z)\to0$ gives $S(z)=S(0)$ for every $z$. Since a constant inner function is unimodular, we absorb $S$ into $\lambda$ and obtain $\Theta=\lambda z^n$.

\end{proof}

The last part of this proof is the fixed-point argument used for zero-free inner functions in \cite{BMS}. We note that the hypothesis $|\alpha|<1$ is essential since the identity preserves every model space, and rotations can preserve additional model spaces.

\section{Invariant Beurling subspaces}\label{sec:beurling}
Before proving Theorem B, we record a basic result about analytic continuation of an $H^2$ function across an arc in $\mathbb{T}$ and the vanishing of its singular measure there. This follows from Garnett's analytic-continuation result
\cite[Chapter II, Theorem~~6.3]{Garnett}.

\begin{lemma}\label{lem:localization}
Let $F=\Theta q\in H^2$ be nonzero, where $\Theta$ is inner
and $q\in H^2$. Let $\mu$ be the singular measure determining
the singular inner factor of $\Theta$.
If $F$ extends analytically across an open arc
$I\subseteq\mathbb T$, then $\mu(I)=0$.
\end{lemma}

\begin{proof}
Let $\nu_F$ and $\nu_q$ denote the positive singular measures
determining the singular inner factors of $F$ and $q$,
respectively. Since $F=\Theta q\ne0$, uniqueness of canonical
factorization gives
\[
  \nu_F=\mu+\nu_q.
\]
Since $F$ extends analytically across $I$,
by Garnett's theorem~[9, Chapter~II,
Theorem~6.3] and its proof, $\nu_F(I)=0$. Positivity of the measures yields $\mu(I)=0$.
\end{proof}

\begin{proof}[Proof of Theorem~\ref{thm:beurling}]
We prove necessity first. Suppose $\Theta\Hp$ is invariant under $C_{\phi_a}$. The quotient criterion \eqref{eq:quotient} provides a Schur function $q$ with
\begin{equation}\label{eq:factoridentity}
F:=\Theta\circ\phi_a=\Theta q.
\end{equation}
The function $F$ is nonzero because $\phi_a$ is non-constant and $\Theta$ is nonzero. The image $\phi_a(\overline\D)$ is the closed disk with center $1-a$ and radius $a$, meeting $\T$ only at $1$. Thus for each $\xi\in\T\setminus\{1\}$, the point $\phi_a(\xi)$ lies in $\D$. In a sufficiently small neighborhood of $\xi$, $\phi_a$ still takes values in $\D$. So the formula $F=\Theta\circ\phi_a$ gives an analytic continuation of $F$ across the boundary arc $I=\T\setminus\{1\}$. If $\mu$ is the singular measure determining $\Theta$, then $\mu(I)=0$ by Lemma~\ref{lem:localization} and hence $\mathrm{supp}(\mu)\subseteq\{1\}$. It follows that $\mu=c\delta_1$, for some constant $c\ge0$. This proves $\Theta=\lambda B S_c$, with $B$ its (possibly constant) Blaschke factor. It remains to prove the multiplicity condition on $B$. At any $w\in\D$, the identity $F=\Theta q$ implies $m_\Theta(w)\le m_F(w)$. The factors $S_c$ and $\lambda$ have no zeros. Moreover, $\phi_a'(w)=a\ne0$, so composition with $\phi_a$ preserves the order of a zero at its image:
\[
m_B(w)=m_\Theta(w)\le m_F(w)
=m_\Theta(\phi_a(w))=m_B(\phi_a(w)).
\]
This proves \eqref{eq:multiplicity}, including when $B$ has infinitely many zeros. For sufficiency, assume $\Theta=\lambda BS_c$ and \eqref{eq:multiplicity}. Since $\phi_a'$ never vanishes, \eqref{eq:multiplicity} is equivalent to $m_B(w)\le m_{B\circ\phi_a}(w)$ for every $w$. The Blaschke criterion \eqref{eq:blaschkecriterion}, followed by \eqref{eq:quotient}, gives a Schur function
\[
q_B=\frac{B\circ\phi_a}{B}.
\]
The singular factor can be treated explicitly. Set $h(z)=(1+z)/(1-z)$, which has positive real part on $\D$. Since $1-\phi_a(z)=a(1-z)$, a calculation yields
\[
h(\phi_a(z))=\frac{h(z)}a+\frac{1-a}{a}.
\]
Therefore
\begin{equation}\label{eq:atomicquotient}
q_c(z):=\frac{S_c(\phi_a(z))}{S_c(z)}
=\exp\left(-\frac{c(1-a)}a\bigl(h(z)+1\bigr)\right).
\end{equation}
The exponent has nonpositive real part, so $|q_c(z)|\le1$. Thus $q_c$ is Schur, also when $c=0$. Multiplying the two factor identities gives $
\Theta\circ\phi_a=\Theta(q_Bq_c).
$
Since $q_Bq_c$ is a Schur function, \eqref{eq:quotient} shows that $\Theta\Hp$ is invariant.
\end{proof}

\begin{corollary}\label{cor:spectrum}
If $\Theta$ is nonconstant and $\Theta\Hp$ is $C_{\phi_a}$-invariant, then
\[
\Sigma(\Theta)=\{1\}.
\]
In particular, either $\Theta=\lambda S_c$ with $c>0$, or its zero set is infinite and has $1$ as its only boundary accumulation point. In the latter case, a factor $S_c$ with $c\ge0$ is still allowed.
\end{corollary}
\begin{proof}
Theorem~\ref{thm:beurling} already places the singular measure at $1$. If distinct zeros $w_j$ converged to $\xi\in\T\setminus\{1\}$, the distinct zeros $\phi_a(w_j)$ would converge to $\phi_a(\xi)\in\D$, contradicting the identity theorem. If a zero $w$ exists, all points
\begin{equation}\label{eq:orbit}
\phi_a^{\circ n}(w)=1-a^n(1-w),\qquad n\ge0,
\end{equation}
are distinct zeros tending to $1$. If there are no zeros, nonconstancy forces $c>0$. Formula~\eqref{eq:spectrum} now proves the assertion.
\end{proof}

\section{Common invariant subspaces and the Ces\`aro operator}\label{sec:cesaro}
Before we prove Theorem C, we establish a result that is of independent interest regarding the common Beurling invariant subspaces of $C_{\phi_a}$ for $0<a<1$. For an inner function $\Theta$, define
\[
A_\Theta=\{a\in(0,1):C_{\phi_a}(\Theta\Hp)\subseteq\Theta\Hp\}.
\]
This set is multiplicatively closed since $\phi_a\circ\phi_b=\phi_{ab}$.
The next result shows that either $A_\Theta=(0,1)$, or it has no accumulation points in $(0,1]$.
\begin{lemma}\label{lem:parameters}
The set $A_\Theta$ has an accumulation point in $(0,1]$ if and only if $\Theta=\lambda S_c$ for some $\lambda\in\T$ and $c\ge0$. In this case $A_\Theta=(0,1)$. If $\Theta$ has a zero, then $A_\Theta\cap[\varepsilon,1)$ is finite $\forall$ $\varepsilon>0$.
\end{lemma}
\begin{proof}
Suppose distinct $a_j\in A_\Theta$ tend to $a_*\in(0,1]$. If $\Theta(w)=0$ for some $w\in\mathbb{D}$, the quotient criterion \eqref{eq:quotient} makes
\[
\phi_{a_j}(w)=1-a_j+a_jw
\]
a zero of $\Theta$ for every $j$. These points are distinct because $w\ne1$, and they converge to $1-a_*+a_*w\in\D$. The identity theorem gives a contradiction. Thus $\Theta$ has no zeros. Applying Theorem~\ref{thm:beurling} to any one of the parameters shows that $\Theta=\lambda S_c$. Conversely, \eqref{eq:atomicquotient} is Schur for every $a\in(0,1)$, so such a $\Theta$ has $A_\Theta=(0,1)$. Finally, if $\Theta$ has a zero and $A_\Theta\cap[\varepsilon,1)$ were infinite, compactness of $[\varepsilon,1]$ would give an accumulation point in $(0,1]$, which we have just ruled out.
\end{proof}

\begin{proof}[Proof of Theorem~\ref{thm:cesaro}]
For part (1), set $M=\Theta\Hp$ in \eqref{eq:cesarocorrespondence}. Its orthogonal complement is $K_\Theta$, so $M$ is $\Ces$-invariant exactly when $K_\Theta$ is invariant under every $C_{\phi_a}$. By Theorem~\ref{thm:model}, invariance under even one of these operators forces $\Theta=\lambda z^n$. Conversely, Theorem~\ref{thm:model} also says that $K_{z^n}$ is invariant under every $C_{\phi_a}$. The correspondence then gives $\Ces$-invariance of $z^n\Hp$. For part (2), set $M=K_\Theta$ in \eqref{eq:cesarocorrespondence} and so $M^\perp=\Theta\Hp$. Thus $K_\Theta$ is $\Ces$-invariant exactly when $A_\Theta=(0,1)$ and so Lemma~\ref{lem:parameters} implies $\Theta=\lambda S_c$. Conversely, \eqref{eq:atomicquotient} and the quotient criterion make $S_c\Hp$ invariant under every $C_{\phi_a}$, and \eqref{eq:cesarocorrespondence} gives invariance of $K_{S_c}$. We next address the assertions about minimality and maximality. Given $0<d<c$, the identity $S_c=S_dS_{c-d}$ yields a strict inclusion
\[
\{0\}\subsetneq K_{S_d}\subsetneq K_{S_c}.
\]
Both model spaces are $\Ces$-invariant by part (2), so $K_{S_c}$ is not minimal. For any $e>c$, the same factorization gives $K_{S_c}\subsetneq K_{S_e}$; hence it is not maximal. Finally, part (1) shows that every invariant Beurling subspace $z^n\Hp$ contains $z^{n+1}\Hp$ and hence none is minimal.
\end{proof}

\section{Cyclic subspaces intersecting $\Theta H^2$ and $K_\Theta$ }
Fix $0<a<1$. For \(0\ne f\in H^2\), write
\[
\mathcal M_f
 =\overline{\operatorname{span}}
   \{C_{\phi_a}^{\,j}f:j\ge0\}
\]
for its cyclic \(C_{\phi_a}\)-invariant subspace. It was shown by Carmo and Noor \cite{CN} that the ISP for complex separable Hilbert spaces has an
affirmative solution if and only if every minimal \(\mathcal M_f\) is
one-dimensional (equivalently $f$ is a $C_{\phi_a}$-eigenvector). In this section we consider what happens when a Beurling or model subspace intersects some \(\mathcal M_f\) non-trivially.  

\begin{proposition}\label{cor:singular-disjoint}
Let \(\Theta\) be a non-constant zero-free inner function such that
\(\Theta H^2\) is \(C_{\phi_a}\)-invariant. If $\mathcal M_f\cap\Theta H^2\neq\{0\}$, then \(\mathcal M_f\) is not a minimal invariant subspace.

\end{proposition}

\begin{proof}
By Theorem B, \(\Theta=\lambda S_c\) for some \(c>0\). Since
\(\mathcal M_f\cap\Theta H^2\ne\{0\}\) and both spaces are
\(C_{\phi_a}\)-invariant, their intersection is a non-zero invariant
subspace of \(\mathcal M_f\). Suppose on the contrary that $\mathcal{M}_f$ is minimal.  This gives
\(\mathcal M_f\subseteq S_cH^2\), and in particular \(f\in S_cH^2\).
This contradicts a result of Carmo and Noor \cite[Theorem 5.5]{CN} that no
non-zero function in \(S_cH^2\) generates a minimal
\(C_{\phi_a}\)-invariant cyclic subspace.
\end{proof}

If $\mathcal{M}_f$ contains a nonzero function analytic at $1$, then it is minimal if and only if dim($\mathcal{M}_f)=1$. This follows from \cite[Theorem 16]{CEN}, since minimality implies \(\mathcal{M}_f=\mathcal{M}_g\) for every nonzero \(g\in\mathcal {M}_f\). The next result characterizes this property through intersection with model spaces whose boundary spectrum omits \(1\).

\begin{theorem}\label{thm:analytic-model-intersection}
Fix $0<a<1$ and $0\ne f\in H^2$. The following are equivalent:
\begin{enumerate}
\item $\mathcal M_f$ contains a nonzero function analytic in a
      neighborhood of $1$;
\item the $C_{\phi_a}$-eigenvector $(1-z)^n\in\mathcal M_f$ for some $n\in\mathbb N_0$;
\item $\mathcal M_f\cap K_\Theta\ne\{0\}$ for some inner function
      $\Theta$ such that $1\notin\Sigma(\Theta)$.
\end{enumerate}

\end{theorem}

\begin{proof}
Suppose (1) holds. Choose a nonzero $g\in\mathcal M_f$ analytic
near $1$. Let $n$ be the order of its zero at $1$, and write $g(z)=(1-z)^n h(z)$
near $1$, where $h$ is analytic and $h(1)=\gamma\ne0$. Since
\[
  \phi_a^{\circ k}(z)=1-a^k(1-z),
\]
the set $\phi_a^{\circ k}(\mathbb D)$ lies in
$\{w:|w-1|<2a^k\}$. Thus, for all sufficiently large $k$,
$h\circ\phi_a^{\circ k}$ is defined on $\mathbb D$ and converges
uniformly there to $\gamma$. For such $k$,
\[
  a^{-kn}C_{\phi_a}^{\,k}g
    =(1-z)^n h\circ\phi_a^{\circ k}.
\]
Moreover,
\[
\begin{aligned}
 \left\|a^{-kn}C_{\phi_a}^{\,k}g
          -\gamma(1-z)^n\right\|_{H^2}
 &\leq \|(1-z)^n\|_{H^2}
       \sup_{z\in\mathbb D}
       \left|h\bigl(\phi_a^{\circ k}(z)\bigr)-\gamma\right|\longrightarrow 0.
\end{aligned}
\]
Because $\mathcal M_f$ is closed and $C_{\phi_a}$-invariant,
$(1-z)^n\in\mathcal M_f$. This proves (2). If (2) holds, choose $\Theta(z)=z^{n+1}$. Then
$(1-z)^n\in K_\Theta$, and $\Sigma(\Theta)=\varnothing$.
Hence (3) holds. Finally, suppose (3) holds and choose
$0\ne g\in\mathcal M_f\cap K_\Theta$. Every function in $K_\Theta$
extends analytically across boundary points outside
$\Sigma(\Theta)$. Since $1\notin\Sigma(\Theta)$, the function $g$ is analytic
near $1$. This proves (1).
\end{proof}

The next example demonstrates the existence of an infinite-dimensional $\mathcal{M}_f$ which contains an eigenvector and yet $\mathcal M_f\cap K_\Theta=\{0\}$ for all inner
      $\Theta$ with $1\notin\Sigma(\Theta)$.

\begin{example}\label{ex:eigenvector-no-regular-model}
Fix $0<a<1$ and $c>0$, and set $g(z)=(1-z)^{2\pi i/\log a}$ and $f=g+S_c$ which belong to $H^2$. Moreover,
$C_{\phi_a}g=g$ and a direct calculation using (1.2) gives
\[
  C_{\phi_a}^{\,n}S_c
    =e^{-c(a^{-n}-1)}S_{c/a^n},
  \qquad
  \bigl\|C_{\phi_a}^{\,n}S_c\bigr\|_{H^2}
    =e^{-c(a^{-n}-1)}\longrightarrow0.
\]
Consequently $C_{\phi_a}^{\,n}f\to g$ in $H^2$, so the eigenvector
$g\in\mathcal M_f$ and $S_c=f-g\in\mathcal M_f$. The space $\mathcal M_f$ is infinite-dimensional since the iterates of $S_c$ are linearly
independent. Indeed, along the radius these iterates are nonzero scalar multiples of \(e^{-c a^{-n}t}\), where \(t=(1+r)/(1-r)\), and exponential functions with distinct rates are linearly independent. We show that $\mathcal M_f$ contains no nonzero function analytic
at $1$. To see this, since $S_cH^2$ is $C_{\phi_a}$-invariant, every
$F\in\mathcal M_f\subset\mathbb Cg+S_cH^2$ can be written as $F=\lambda g+S_ch$ for $h\in H^2$ (notice that the inclusion holds because the sum of a finite dimensional subspace with a closed subspace remains closed). From the $H^2$ point-evaluation estimate we obtain
\[
  |S_c(r)h(r)|
  \leq \frac{\|h\|_{H^2}}{\sqrt{1-r^2}}\,
       \exp\!\left(-c\,\frac{1+r}{1-r}\right),
  \qquad 0<r<1
\]
and in particular that
\begin{align}\label{Example: core estimate}
    \lim\limits_{r \to 1^-} |S_c(r)h(r)|  = 0.
\end{align}
On the other hand
\[
  g(1-a^n)=1,\qquad
  g(1-a^{n+1/2})=-1,
\]
which implies that $\lim\limits_{r \to 1^-}g(r)$ does not exist. Suppose some $0\ne F\in\mathcal M_f$ extends analytically
across $1$. Since $S_ch(r)\to0$, the function
$\lambda g(r)=F(r)-S_ch(r)$ has a finite radial limit.
The two displayed sequences force $\lambda=0$.
Thus $F=S_ch\ne0$ extends analytically across $1$,
contradicting Lemma~4.1 because the singular measure
of $S_c$ assigns mass $c>0$ to $1$. This implies that $F=S_ch$ is analytic at $1$ which contradicts Lemma \ref{lem:localization}. Therefore $\mathcal M_f$ contains no nonzero function analytic at $1$ and hence $\mathcal M_f\cap K_\Theta=\{0\}$ for all inner $\Theta$ with $1\notin\Sigma(\Theta)$ by Theorem \ref{thm:analytic-model-intersection}.
\end{example}

\section{On AI assistance}

This paper revises and extends the authors’ preprint of the same title \cite{ENpreprint}, which remains available on arXiv and contains the original model-space classification for \(\phi_a\), the Beurling-space dichotomy, and their Cesàro consequences. Much of that work formed part of the first author’s doctoral thesis. The AI models GPT Sol~6 and Sol~6.1 were used in preparing this revision, to help verify and streamline the arguments, and
identify appropriate references for attribution and credit. The authors
retain responsibility for the mathematical content and references.

\section{Acknowledgments} The original work \cite{ENpreprint} formed part of the doctoral thesis of the first author, partially supported by the Conselho Nacional de Desenvolvimento Cient\'ifico e Tecnol\'ogico (CNPq), Brazil, under the supervision of the second author. Currently, Ben-Hur Eidt is supported by the São Paulo Research Foundation (FAPESP), Brasil. Process Number 2024/14727-4.

\end{document}